\documentclass[12pt]{amsart}
\usepackage{epsfig, amsmath, amssymb, latexsym,  amscd}
\usepackage{color}

\def\cocoa{{\hbox{\rm C\kern-.13em o\kern-.07em C\kern-.13em o\kern-.15em A}}}

\newtheorem{theorem}{Theorem}[section]

 \newtheorem{proposition}[theorem]{Proposition}
 
 \theoremstyle{definition}
 \newtheorem{definition}[theorem]{Definition}
 \theoremstyle{remark}
 \newtheorem{remark}[theorem]{Remark}

  \definecolor{MyDarkGreen}{cmyk}{0.7,0,1,0}

\author[J.~Migliore]{Juan Migliore}
\address{Department of Mathematics, University of Notre Dame, Notre Dame, IN 46556}
\email{migliore.1@nd.edu}

\author[F.\ Zanello]{Fabrizio Zanello}
\address{Department of Mathematical Sciences, Michigan Tech, Houghton, MI 49931}
\email{zanello@mtu.edu}

\title[Stanley's nonunimodal Gorenstein $h$-vector is optimal]{Stanley's nonunimodal Gorenstein $h$-vector is optimal}

\begin{document}

\begin{abstract} 
We classify all possible $h$-vectors of graded artinian Gorenstein algebras in socle degree 4 and codimension $\leq 17$, and in socle degree 5 and codimension $\leq 25$. We obtain as a consequence that the least number of variables allowing the existence of a nonunimodal Gorenstein $h$-vector is 13 for socle degree 4, and 17 for socle degree 5. 
In particular, the smallest nonunimodal Gorenstein $h$-vector is $(1,13,12,13,1)$, which was constructed by Stanley in  his 1978 seminal paper on level algebras. This solves a long-standing open question in this area. All of our results are characteristic free.
\end{abstract}


\maketitle

\section{Introduction}

It is a long-standing open problem in combinatorial commutative algebra to provide a classification of all possible artinian Gorenstein $h$-vectors. Since producing an explicit characterization seems hopeless for algebras of codimension $\geq 5$ (i.e., in five or more variables), much research over the years has been devoted to trying to determine conditions under which certain Gorenstein $h$-vectors can be nonunimodal (see, as a  nonexhaustive list, \cite{AS, BI, B, BL, BZ, BE, H, IS, MNZ1, MNZ3, MNZ2, SS, stanley,St2,zanello1,zanello3}). Recall that a sequence of integers is \emph{unimodal} if it does not strictly increase after a strict decrease. (For the importance of unimodality in algebra, combinatorics and related areas, see for instance the two classical surveys of Stanley \cite{St0} and Brenti \cite{Bre}.) In particular, today we know that nonunimodal Gorenstein $h$-vectors exist in any codimension $r\ge 5$ \cite{BI}, while when $r=3$, all Gorenstein $h$-vectors need to be unimodal (in fact, a complete characterization is known; see  Stanley \cite{stanley}, whose proof was based on the structure theorem of Buchsbaum and Eisenbud \cite{BE}, and then the second author \cite{zanello1} for a combinatorial proof). It is still open whether nonunimodal Gorenstein $h$-vectors exist when $r=4$ (see \cite{IS,MNZ3,SS} for some progress, mainly over a field of characteristic zero).

The first example of a nonunimodal Gorenstein $h$-vector, namely $(1,13,12,13,1)$, which has {socle degree} 4 (i.e., length 5), was produced by Stanley \cite{stanley}, using the technique of \emph{trivial extensions}, also introduced in that paper and useful in some of our proofs here. Notice that  4 is the smallest socle degree allowing the existence of a nonunimodal Gorenstein $h$-vector, because Gorenstein $h$-vectors are symmetric. Following Stanley \cite{stanley}, we denote  by $f(r)$ the least positive integer such that $(1,r,f(r),r,1)$ is a Gorenstein $h$-vector. These authors in collaboration with Nagel showed in \cite{MNZ1} that $f(r)$ is asymptotic to $(6r)^{2/3}$ for $r\rightarrow \infty$, thus solving a long-standing conjecture of Stanley \cite{St2}. Further, it was proven in \cite{zanello3} that $(1,r,a,r,1)$ is a  Gorenstein $h$-vector for all $a=f(r),f(r)+1,\dots,\binom{r+1}{2}$. However, except for special values of $r$, in general the integer $f(r)$ is not known. Until now, a basic open problem in this area was to determine the least integer $r$ such that $f(r)<r$, i.e., the least codimension $r$ allowing the existence of nonunimodal Gorenstein $h$-vectors in socle degree 4. (It is easy to see that if  a nonunimodal Gorenstein $h$-vector exists in codimension $r$, then one must also exist in all larger codimensions.)

Several incremental results, both characteristic-free or with assumptions on the characteristic (see for instance \cite{AS,BZ,MNZ1, stanley}) finally  left open only the existence of $(1,12,11,12,1)$ as a possible  Gorenstein $h$-vector, in order to show that  all Gorenstein $h$-vectors of socle degree 4 and codimension $r\le 12$ must be unimodal. In this paper we settle this problem, by showing that $(1,12,11,12,1)$ cannot be Gorenstein in any characteristic. Our method is a refinement of the approach begun by Stanley and improved upon by the authors mentioned above. As a consequence, this also proves that Stanley's $(1,13,12,13,1)$ original example is the smallest possible. Furthermore, using our method we  obtain a full classification of all Gorenstein $h$-vectors with socle degree 4 and codimension $\leq 17$. 

In socle degree 5, the work of several authors culminated in Theorem 3.3 of \cite{AS}, which showed that all socle degree 5 Gorenstein $h$-vectors of codimension $\leq 15$ are unimodal. The authors of \cite{AS} also pointed out that $(1,18,16,16,18,1)$ is a Gorenstein $h$-vector, while $(1,16,14,14,16,1)$ is not. This left open the existence of the following three $h$-vectors of codimension $r\le 17$: $H^{(1)}=(1,16,15,15,16,1)$, $H^{(2)}=(1,17,15,15,17,1)$, and $H^{(3)}=(1,17,16,16,17,1)$. In this note, as an application of our method, we prove that $H^{(1)}$ and $H^{(2)}$ cannot be Gorenstein, while we see using trivial extensions that $H^{(3)}$ is. Therefore, the least codimension allowing the existence of a nonunimodal Gorenstein $h$-vector in socle degree 5 is 17. Furthermore, we obtain a complete classification of Gorenstein $h$-vectors for socle degree 5 in codimension $\leq 25$.

Our results in this paper are  characteristic free.

\section{Preliminary facts} 

In this section, we briefly gather the main definitions and results needed in this paper. Let $A=\bigoplus_{i \geq 0} A_i=R/I$ be a standard graded $k$-algebra, where $R = k[x_1, \dots, x_r]$, $I$ is a homogeneous ideal of $R$, and $k$ is any infinite field.   The {\em Hilbert function} of $A$ in degree~$i$ is $h_{R/I}(i) = h_i = \dim_k A_i$. When $A$ is \emph{artinian} (i.e., it has Krull-dimension zero), the Hilbert function of $A$ is eventually zero and can be identified with its \emph{$h$-vector}, namely $h=(h_0=1,h_1,\dots, h_{e-1},h_e>0)$.   Hence we will sometimes use the two terms interchangeably. The integer $e$ is called the \emph{socle degree} of $A$, and since we may assume without loss of generality that $I$ contains no nonzero forms of degree 1, the number of variables $r$ equals $h_1$ and is called the \emph{codimension} of $A$. 

Define the \emph{socle} of an  algebra $A$ of depth zero (for us, usually artinian) as the annihilator of the maximal homogeneous ideal $\overline{m}=(\overline{x_1},\dots,\overline{x_r})\subseteq A$, namely $soc(A)=\lbrace a\in A {\ } \mid {\ } a\overline{m}=0\rbrace $. The \emph{socle vector} of $A$ is $s=(s_0,s_1,\dots,s_e)$, where $s_i=\dim_k soc(A)_i$. Notice that $s_0=0$ and $s_e=h_e>0$.

If $s=(0,\dots,0,s_e=t)$, we say that $A$ is \emph{level (of type $t$)}. In particular, if $t=1$, $A$ is \emph{Gorenstein}. With a slight abuse of notation, we will refer to an $h$-vector as Gorenstein (or level) if it is the $h$-vector of a Gorenstein (or level) artinian algebra.

The following three basic results are due to Macaulay, Gotzmann, and Green; before stating them, we need to recall the following definition.

\begin{definition} Let $n$ and $i$ be positive integers. The {\em $i$-binomial expansion of $n$} is 
\[
n_{(i)} = \binom{n_i}{ i}+\binom{n_{i-1}}{ i-1}+...+\binom{n_j}{ j},
\]
 where $n_i>n_{i-1}>...>n_j\geq j\geq 1$. Such an expansion always exists and is unique (see, e.g.,  \cite{BH}, Lemma 4.2.6). Following \cite{BH}, we define, for any integers $a$ and $b$,
$$
(n_{(i)})_{a}^{b}=\binom{n_i+b}{ i+a}+\binom{n_{i-1}+b}{ i-1+a}+...+\binom{n_j+b}{ j+a},
$$
where we set $\binom{m}{ c}=0$ whenever $m<c$ or $c<0$.
\end{definition}

\begin{theorem}\label{MGG} Let $A=R/I$ be a standard graded $k$-algebra, and $L \in A$  a general linear form (according to the Zariski topology). Denote by $h_d$  the degree $d$ entry of the Hilbert function of $A$ and by $h_d^{'}$ the degree $d$ entry of the Hilbert function of $A/(L)$. Then:

\begin{itemize}

\item[(i)] {\em (Macaulay)} $$h_{d+1}\leq ((h_d)_{(d)})_{+1}^{+1}.$$\indent

\item[(ii)] {\em (Gotzmann)} 
If $h_{d+1} = ((h_d)_{(d)})^{+1}_{+1}$ and $I$ is generated in degrees $\leq d+1$, then 
\[
h_{d+s} = ((h_d)_{(d)})^s_s \hbox{\hspace{.3cm} for all $s \geq 1$}.
\]

\item[(iii)] {\em (Green)} $$h_d^{'}\leq ((h_d)_{(d)})_{0}^{-1}.$$
\end{itemize}
\end{theorem}

\begin{proof}
For part (i) see \cite{BH}, Theorem 4.2.10, or \cite{Ma}. For (ii) see \cite{BH}, Theorem 4.3.3, or \cite{Go}. For (iii) see \cite{Gr}, Theorem 1. 
\end{proof}

A sequence of nonnegative integers $h = (1,h_1,h_2, \dots , h_i ,\dots )$ is said to be an {\em $O$-sequence} if it satisfies Macaulay's theorem for all $i$.  Recall that when $A$ is artinian and Gorenstein, then its Hilbert function (or $h$-vector) is a finite, symmetric $O$-sequence. 

We recall a useful theorem proven in \cite{Za4}, Theorem 3.5.

\begin{theorem}\label{F socle}
Let $h_{d-1}$, $h_d$ and $h_{d+1}$ be three integers such that $((h_d)_{(d)})^{-1}_{-1} = h_{d-1}$ and $((h_d)_{(d)})^{+1}_{+1} = h_{d+1}$. Suppose that $h_{d-1} + \alpha$, $h_d$ and $h_{d+1}$, for some integer $\alpha > 0$, are the entries of degree $d-1$, $d$ and $d+1$ of the $h$-vector of an  algebra $A$. Then $A$ has depth zero and  an $\alpha$-dimensional socle in degree $d-1$.
\end{theorem}

We also need the following  decomposition theorem shown in \cite{GHMS}, Lemma 2.8 and Theorem 2.10.

\begin{theorem}\label{decomp}
Let  $h=(1,h_1,\dots,h_e=t)$ be the $h$-vector of a level algebra $A$ of type $t$, and let $h'=(1,h_1',\dots,h_e'=t-1)$ be the $h$-vector of any level quotient of $A$ of type $t-1$ and the same socle degree, $e$. Then the reverse of the difference of $h$ and $h'$, namely $(1,h_{e-1}-h'_{e-1},h_{e-2}-h'_{e-2}, \dots)$, is an $O$-sequence.
\end{theorem}

Finally, the following result will be useful.

\begin{theorem} \label{increase ab}
If $(1,r,a,r,1)$ is a Gorenstein $h$-vector, then so are $(1,r,b,r,1)$ for each $a \leq b \leq \binom{r+1}{2}$, and $(1,r+1,a+1,r+1,1)$. Similarly, if $(1,r,a,a,r,1)$ is Gorenstein, then so are $(1,r,b,b,r,1)$ for each $a \leq b \leq \binom{r+1}{2}$, and $(1,r+1,a+1,a+1,r+1,1)$.
\end{theorem}

\begin{proof}
This follows from standard inverse system techniques. See for example \cite{iarrobino, zanello3}.
\end{proof}

\section{Main Results}

In socle degree 4, the only codimension in which it is unknown whether a nonunimodal Gorenstein $h$-vector exists is 12. We now settle this question.

\begin{proposition} \label{1 12 11 12 1}
The $h$-vector $(1,12,11,12,1)$ is not  Gorenstein.
\end{proposition}

\begin{proof}
Let $R = k[x_1,\dots,x_{12}]$ and suppose that $I \subset R$ were a Gorenstein ideal with $h$-vector $(1,12,11,12,1)$.  Let $J = ( I_{\leq 3} )$, the ideal generated by the components of $I$ in degrees 2 and 3. Let $L$ be a general linear form,  $S = R/(L)$, and  $\bar J = \frac{J,L}{(L)} \subset S$ (and similarly for $\bar I$). Employing the exact sequence
\[
0 \rightarrow R/(I:L)(-1) \rightarrow R/I \rightarrow S/\bar I \rightarrow 0
\]
and applying Theorem \ref{MGG}, (iii) (Green's theorem), Theorem \ref{MGG}, (i) (Macaulay's theorem) and the symmetry of the  $h$-vector of the Gorenstein algebra $R/(I:L)$, it is not hard to check that the rows of the following diagram represent the only possible values for the Hilbert functions of $R/I$, $R/(I:L)(-1)$ and $S/\bar I$ respectively:

\begin{center}

\begin{tabular}{ccccc}
1 & 12 & 11 & 12 & 1 \\
& 1 & 8 & 8 & 1 \\ \hline
1 & 11 & 3 & 4
\end{tabular}

\end{center}

\noindent In degrees $\leq 3$ the above is also correct replacing $I$ by $J$. Because $S/\bar J$ has maximal growth (according to Macaulay's theorem) from degree 2 to degree 3 and $\bar J$ has no new generators in degree $\geq 4$, we obtain from Theorem \ref{MGG}, (ii) (Gotzmann's theorem) that $h_{S/\bar J}(t) = t+1$ for all $t \geq 2$. Thus $\bar J$ is the saturated ideal of a line in $\mathbb P^{10}$ in all degrees $\geq 2$. It follows that, up to saturation, $J$ is the ideal of the union  in $\mathbb P^{11}$ of a plane and a finite number, say $m$, of points (possibly embedded) --- see for instance \cite{BGM}. 

\medskip

\noindent {\em Claim:} \ $15 +m \leq h_{R/J}(4) \leq 17$.  In particular,  $0 \leq m \leq 2$.

The upper bound is given by Macaulay's theorem, since $h_{R/J}(3) = 12$. The lower bound follows from the fact that $h_{R/J}(4) \geq h_{R/J^{sat}}(4) = \binom{6}{2} + m $.

\medskip

\noindent \underline{Case 1}: $m = 2$. Then $h_{R/J} (4) = 17$. We have  $(12_{(3)})^{-1}_{-1} = 8$ and $(12_{(3)})^{+1}_{+1} = 17$, while $h_{R/J}(2) = 11 > 8$. Hence Theorem \ref{F socle} gives that $R/J$ has a 3-dimensional socle  in degree 2, contradicting the assumption that $R/I$ is Gorenstein. 

\medskip

\noindent \underline{Case 2}: $m = 1$. If $h_{R/J} (4) = 17$ then the same argument as in Case 1 applies. If $h_{R/J} (4) = 16$ then $J$ is saturated in degree 4 but not in degree 3 (since $h_{R/ J^{sat}}(3) = 10+1 = 11 < 12 = h_{R/J}(3)$). Therefore $R/J$ has a socle element in degree 3. Since $R/I$ is a quotient of $R/J$, it follows that $R/I$ also has a socle element in degree 3, contradicting the Gorenstein assumption. 

\medskip

\noindent \underline{Case 3}: $m = 0$. Now $15 \leq h_{R/J}(4) \leq 17$. If $h_{R/J}(4) = 17$, the same argument as in Case 1 applies. If $h_{R/J}(4) = 15$, again $J$ is saturated in degree 4 so the same argument as in Case 2 applies. If $h_{R/J}(4) = 16$, the result follows from Theorem \ref{decomp}, since the truncated Hilbert function of the plane is $(1,3,6,10,15)$, which is level, and  $(16-15, 12-10, 11-6, 12-3) = (1,2,5,9)$ is not an $O$-sequence. 
\end{proof}

Recall that $(1,13,12,13,1)$ was shown by Stanley \cite{stanley} to be a Gorenstein $h$-vector. From Stanley's example one can use Theorem \ref{increase ab} to construct Gorenstein $h$-vectors of the form $(1,r,r-1,r,1)$, for all $r \geq 13$. Furthermore, it is well known that all unimodal symmetric $h$-vectors with $h_2 \leq \binom{r+1}{2}$ are Gorenstein in socle degrees 4 and 5 (for instance it follows from Theorem \ref{increase ab}). Thus, we obtain the following characterization.

\begin{theorem} \label{1st main}

The Gorenstein $h$-vectors $(1,r,h_2,r,1)$ of socle degree 4 and codimension $\leq 17$ are precisely the ones with $r \leq h_2 \leq \binom{r+1}{2}$, together with $(1,13,12,13,1)$, $(1,14,13,14,1)$, $(1,15,14,15,1)$, $(1,16,15,16,1)$ and $(1,17,16,17,1)$. In particular, nonunimodal Gorenstein $h$-vectors of socle degree 4 and codimension $r$ exist if and only if $r \geq 13$. 

\end{theorem}

\begin{proof}
Thanks again to Theorem \ref{increase ab}, the only missing ingredient is the non-existence of \linebreak $(1,17,15,17,1)$. This can be shown using the same ideas as above, except that instead of a plane we obtain a quadric surface. We sketch the argument.

Using symmetry, Green's theorem and Macaulay's theorem, the following diagram represents the only possible values for the Hilbert functions of $R/I$, $R/(I:L)$ and $R/(I,L)$:

\begin{center}

\begin{tabular}{ccccc}
1 & 17 & 15 & 17 & 1 \\
& 1 & 10 & 10 & 1 \\ \hline
1 & 16 & 5 & 7
\end{tabular}

\end{center}

\noindent The last two values on the bottom row represent maximal growth according to Macaulay's theorem, so setting $J = ( I_{\leq 3} )$, we see that $J$ defines a quadric surface together with, say, $m$ points (possibly embedded). Taking into account the maximal possible growth, we see that $25 \leq h_{R/J}(4) \leq 26$. If $h_{R/J}(4) = 25$ then $m = 0$ and $J$ is saturated in degree 4, while $J$ is not saturated in degree 3. Thus $R/J$, and hence $R/I$, has socle in degree 3 and we are done. If $h_{R/J}(4) = 26$ then the growth is maximal from degree 3 to degree 4, so again using Theorem \ref{F socle}   we obtain socle in degree 2. This completes the proof.
\end{proof}

\begin{remark}
As a result of Theorem \ref{1st main}, we see that Stanley's original $(1,13,12,13,1)$ example is the smallest possible, as measured by the codimension among socle degree 4 Gorenstein $h$-vectors, and as it can easily be verified, also as measured by the dimension as a $k$-vector space of the Gorenstein algebra.
\end{remark}

We now turn to socle degree 5. 
As noted earlier, in order to show that 17 is the least codimension allowing a nonunimodal Gorenstein $h$-vector in socle degree 5, we need only determine the Gorensteinness (or not) of $H^{(1)}=(1,16,15,15,16,1)$, $H^{(2)}=(1,17,15,15,17,1)$, and $H^{(3)}=(1,17,16,16,17,1)$. We now solve this problem.

Suppose first that $R/I$ were a Gorenstein algebra with $h$-vector $(1,17,15,15, 17,1)$, and let $J = ( I_{\leq 4} )$. Let $L$, $S$, $\bar I$ and $\bar J$ be as before, and we use a diagram similar to the one above. Using Green's theorem and the values of $h_{R/I}$, we obtain the following upper bounds for $h_{S/\bar I}$, and consequently lower bounds for $h_{R/(I:L)}(-1)$. 

\begin{center}

\begin{tabular}{ccccccccc}
& 1 & 17 & 15 & 15 & 17 & 1 \\
$\geq$ && 1 & 5 & 9 & 12 & 1 \\ \hline
$\leq$ & 1 & 16 & 10 & 6 & 5
\end{tabular}

\end{center}

\noindent But $R/(I:L)$ is Gorenstein, so the middle row has to be symmetric. Thus $h_{R/(I:L)}(1) \geq 12$ and we obtain

\begin{center}

\begin{tabular}{ccccccccc}
& 1 & 17 & 15 & 15 & 17 & 1 \\
$\geq$ && 1 & 12 & 9 & 12 & 1 \\ \hline
$\leq$ & 1 & 16 & 3 & 6 & 5
\end{tabular}

\end{center}

\noindent This forces

\begin{center}

\begin{tabular}{ccccccccc}
& 1 & 17 & 15 & 15 & 17 & 1 \\
$\geq$ && 1 & 12 & 11 & 12 & 1 \\ \hline
$\leq$ & 1 & 16 & 3 & 4 & 5
\end{tabular}

\end{center}

\noindent We saw in Proposition \ref{1 12 11 12 1} that no Gorenstein algebra exists with Hilbert function given by the middle row. If the middle row takes higher values (preserving symmetry), we reach a violation of maximal growth on the bottom row. Hence $(1,17,15,15,17,1)$ is not Gorenstein.

We now show that the $h$-vector $(1,16,15,15,16,1)$ is not Gorenstein. Suppose that $R/I$ were a Gorenstein algebra with this Hilbert function, and let $J = ( I_{\leq 4} )$. One uses an argument very similar to the one above, together with the fact that if $(1, 11, b, 11, 1)$ is Gorenstein then $b \geq 11$, to show that the only possible diagram for $R/I$ is

\begin{center}

\begin{tabular}{ccccccccc}
1 & 16 & 15 & 15 & 16 & 1 \\
& 1 & 11 & 11 & 11 & 1 \\ \hline
1 & 15 & 4 & 4 & 5
\end{tabular}

\end{center}

\noindent From Macaulay's theorem and our observation above about the saturation, we obtain $21 \leq h_{R/J}(5) \leq 22$. Analogously to the argument above, $J^{sat}$ is the ideal of a plane plus $m$ points (possibly embedded), where $m = 0$ or 1. If $h_{R/J}(5) = 21$ then $m=0$, $J$ is saturated in degree 5, and $R/I$ has socle in degree 4. If $h_{R/J}(5) = 22$ then it follows from maximal growth and Gotzmann's theorem that 
$m=1$ and $J$ is saturated both in degree 5 and degree 4. However, it is not saturated in degree 3 so we obtain socle in degree 3. Either way, we obtain a contradiction to the Gorenstein assumption. 

On the other hand, we observe that $(1,17,16,16,17,1)$ is Gorenstein, by constructing it with trivial extensions, just as Stanley \cite{stanley} used to construct his famous example $(1,13,12,13,1)$. Indeed, it is enough to apply trivial extensions to the level $h$-vector $(1,3,6,10,14)$, so
\[
\begin{array}{ccccccccc}
1 & 3 & 6 & 10 & 14 \\
& 14 & 10 & 6 & 3 & 1 \\ \hline
1 & 17 & 16 & 16 & 17 & 1
\end{array}
\]
gives the desired Gorenstein $h$-vector.

We  can thus show the following characterization  in socle degree 5, again using the fact that in socle degree $\leq 5$ we know precisely which are the unimodal Gorenstein $h$-vectors.

\begin{theorem} \label{2nd main}
The Gorenstein $h$-vectors $(1,r,h_2,h_2,r,1)$ of socle degree 5 and codimension $r\leq 25$ are precisely the ones with $r \leq h_2 \leq \binom{r+1}{2}$, together with  $(1,r,r-1,r-1,r,1)$ for $17 \leq r \leq 25$ and $(1,r,r-2,r-2,r,1)$ for $18 \leq r \leq 25$.  In particular, nonunimodal Gorenstein $h$-vectors of socle degree 5 and codimension $r$ exist if and only if $r \geq 17$. 
\end{theorem}

\begin{proof}
Since $(1,17,16,16,17,1)$ is Gorenstein but $(1,16,15,15,16,1)$ is not, we obtain the result about $(1,r,r-1,r-1,r,1)$. 
It was shown in \cite{AS} that $(1,18,16,16,18,1)$ is Gorenstein, so using that fact and Theorem \ref{increase ab} we see that the same is true for $(1,r,r-2, r-2,r,1)$ when $r \geq 18$. On the other hand, we saw above that $(1,17,15,15,17,1)$ is not Gorenstein. It remains to prove that $(1,r,r-3,r-3,r,1)$ is not Gorenstein for $r \leq 25$.  Again by Theorem \ref{increase ab}, it is enough to show that $(1,25,22,22,25,1)$ is not Gorenstein. 

Using arguments as above (combining Green's theorem, Macaulay's theorem and symmetry), together with Theorem \ref{1st main}, we obtain

\begin{center}

\begin{tabular}{ccccccccc}
& 1 & 25 & 22 & 22 & 25 & 1 \\
 && 1 & 16 & 15 & 16 & 1 \\ \hline
 & 1 & 24 & 6 & 7 & 9
\end{tabular}

\end{center}

\noindent as the only possible decomposition. The growth from degree 3 to degree 4 is again maximal on the third line, giving a conic, so for $R/J$ we obtain a quadric surface plus possibly a finite number of points, and proceed as before. This completes the proof.
\end{proof}

\begin{remark}
\begin{itemize}

\item[(i)] For socle degree 4 we now know all possible Gorenstein $h$-vectors  that are unimodal or of the form $(1,r,r-1,r,1)$. We saw that the first open case for $(1,r,r-2,r,1)$ is $(1,18,16,18,1)$. One can show that $(1,20,18,20,1)$ comes from trivial extensions using the level $h$-vector $(1,4,9,16)$, so the only other open case is $(1,19,17,19,1)$. For the  cases $r-h_2 = 3$ and $r-h_2 = 4$, using the level $h$-vectors $(1,4,10,19)$ and $(1,4,10,20)$, we obtain,  with trivial extensions, the Gorenstein $h$-vectors $(1,23,20,23,1)$ and $(1,24,20,24,1)$. When $r-h_2 = 3$ we can show that $(1,18,15,18,1)$ does not exist, leaving four open cases. When $r-h_2 = 4$, the same arguments as above rule out $r \leq 22$, and a slightly more subtle argument rules out $r = 23$. It follows that $r - h_2 = 4$ exists if and only if $r \geq 24$.

\medskip

\item[(ii)] For socle degree 5, an obstacle to settling the case $r = 26$ is the fact that we do not know whether $(1,18,16,18,1)$ is Gorenstein for the second row. However, notice that $(1,4,9,16,25)$ and $(1,4,9,16,24)$ are level $h$-vectors, which give by trivial extensions the Gorenstein $h$-vectors $(1,29,25,25,29,1)$ and $(1,28,25,25,28,1)$. Combining this with Theorem \ref{2nd main}, we see that when it comes to  the $h$-vectors of the form $(1,r,r-3,r-3,r,1)$, the only open cases are $(1,26,23,23,26,1)$ and $(1,27,24,24,27,1)$. As for the $h$-vectors of the form $(1,r,r-4,r-4,r,1)$, an argument using cases, similar to that given in Proposition \ref{1 12 11 12 1} (but using a quadric instead of a plane), shows that $(1,27,23,23,27,1)$ does not exist.  (Interestingly, to make this conclusion we have to allow the possibility that $(1,18,16,18,1)$ is Gorenstein, but we show that even if it were to exist, we can complete the argument.) Thus the only open case is $(1,28,24,24,28,1)$.
\end{itemize}
\end{remark}

\begin{remark}\label{shin remark}
In writing this paper, the authors made a conscious decision not to attempt to obtain more general results for higher socle degree, and to rather focus on settling the long-standing open problem that Stanley's original nonunimodal $h$-vector $(1,13,12,13,1)$ is the smallest possible. However, it is clear that the methods that have been brought to bear on the general classification problem addressed here can also be used to prove results in higher socle degree. Indeed, Y.S. Shin has kindly informed us that, after receiving our preprint, he and J. Ahn can now extend Theorem 2.8 of \cite{AS} to a similar statement which includes our Proposition \ref{1 12 11 12 1} as a special case. 
\end{remark}

\section{Acknowledgements} We are grateful to A. Iarrobino for useful comments on the history of the problem, and to Y.S. Shin (see Remark \ref{shin remark}). This work was done while both authors were partially supported by  Simons Foundation grants  (\#309556 for Migliore, \#274577 for Zanello).



\begin{thebibliography}{llll}
\bibitem{AS} J. Ahn and Y.S. Shin: \emph{On Gorenstein sequences of socle degrees 4 and 5}, J. Pure Appl. Algebra  \textbf{217} (2013), 854--862.

\bibitem{BI}  D. Bernstein and A. Iarrobino:  \emph{A nonunimodal graded Gorenstein Artin algebra in codimension five}, Comm. Algebra \textbf{20} (1992), no. 8, 2323--2336.

\bibitem{BGM} A. Bigatti, A.V. Geramita and J. Migliore: {\em Geometric consequences of extremal behavior in a theorem of Macaulay}, Trans. Amer. Math. Soc. {\bf 346} (1994), no. 1, 203--235. 

\bibitem{B} M. Boij: {\em Graded Gorenstein Artin algebras whose Hilbert functions have a large number of valleys}, Comm. Algebra {\bf 23} (1995), no. 1, 97--103.

\bibitem{BL} M. Boij and D. Laksov: {\em Nonunimodality of graded Gorenstein Artin algebras}, Proc. Amer. Math. Soc. {\bf 120} (1994), no. 4, 1083--1092.

\bibitem{BZ}  M. Boij and F. Zanello: \emph{Some algebraic consequences of Green's hyperplane restriction theorems}, J. Pure Appl. Algebra \textbf{214} (2010), no. 7, 1263--1270.

\bibitem{Bre} F. Brenti: \emph{Log-concave and unimodal sequences in algebra, combinatorics, and geometry: an update}, in ``Jerusalem Combinatorics '93,'' Contemporary Math. \textbf{178} (1994), 71--89.

\bibitem{BH} W. Bruns and J. Herzog: ``Cohen-Macaulay rings,'' Cambridge Studies in Advanced Math. \textbf{39}, Revised Ed. (1998), Cambridge, U.K..

\bibitem{BE} D. Buchsbaum and D. Eisenbud: {\em Algebra structures for finite free resolutions, and some structure theorems for ideals of codimension 3}, Amer. J. Math. {\bf 99} (1977), no. 3, 447--485.

\bibitem{CI} Y. Cho and A. Iarrobino: \emph{Hilbert functions of level algebras}, J. Algebra \textbf{241} (2001), 745--758.

\bibitem{GHMS}  A.V. Geramita, T. Harima,  J. Migliore, and Y.S. Shin: ``The Hilbert function of a level algebra,'' Mem. Amer. Math. Soc. \textbf{186} (2007), no. 872.

\bibitem{Go} G. Gotzmann: \emph{Eine Bedingung f\"ur die Flachheit und das Hilbertpolynom cines graduierten Ringes}, Math. Z. \textbf{158} (1978), no. 1, 61--70.

\bibitem{Gr} M. Green: \emph{Restrictions of linear series to hyperplanes, and some results of Macaulay and Gotzmann}, Algebraic curves and projective geometry (1988), 76--86, Trento; Lecture Notes in Math. \textbf{1389} (1989), Springer, Berlin.

\bibitem{H} T. Harima: {\em Characterization of Hilbert functions of Gorenstein Artin algebras with the weak Stanley property}, Proc. Amer. Math. Soc. {\bf 123} (1995), no. 12, 3631--3638.

\bibitem{iarrobino} A. Iarrobino: {\em Compressed Algebras: Artin algebras having given socle degrees and maximal length}, Trans. Amer. Math. Soc. {\bf 285} (1984) 337--378.

\bibitem{IS} A. Iarrobino and H. Srinivasan: {\em Artinian Gorenstein algebras of embedding dimension four: components of $PGor(H)$ for $H=(1,4,7,...,1)$},  J. Pure Appl. Algebra {\bf 201} (2005), no. 1-3, 62--96.

\bibitem{Ma} F.H.S. Macaulay: \emph{Some property of enumeration in the theory of modular systems}, Proc. Lond. Math. Soc. \textbf{26} (1) (1927), 531--555.

\bibitem{MNZ1} J. Migliore, U. Nagel and F. Zanello: \emph{On the degree two entry of a Gorenstein $h$-vector and a conjecture of
Stanley}, Proc. Amer. Math. Soc. \textbf{136} (2008), no. 8, 2755--2762.

\bibitem{MNZ3} J. Migliore, U. Nagel and F. Zanello: \emph{A characterization of Gorenstein Hilbert functions in codimension four with small initial degree}, Math. Res. Lett. \textbf{15} (2008), no. 2, 331--349.

\bibitem{MNZ2}  J. Migliore, U. Nagel and F. Zanello: {\em Bounds and asymptotic minimal growth for Gorenstein Hilbert functions}, J. Algebra {\bf 321} (2009), no. 5, 1510--1521.

\bibitem{SS} S. Seo  and H. Srinivasan: {\em On unimodality of Hilbert functions of Gorenstein Artin algebras of embedding dimension four}, Comm. Algebra {\bf 40} (2012), no. 8, 2893--2905.

\bibitem{stanley} R.P. Stanley: \emph{Hilbert functions of graded algebras}, Adv. Math. \textbf{28} (1978), 57--83.

\bibitem{St0} R.P. Stanley: \emph{Log-concave and unimodal sequences in algebra, combinatorics, and geometry}, Ann. New York Acad. Sci. \textbf{576} (1989), 500--535.

\bibitem{St2} R.P. Stanley: ``Combinatorics and Commutative Algebra,'' Second Ed., Progress in Mathematics \textbf{41} (1996), Birkh\"auser, Boston.

\bibitem{Za4} F. Zanello: \emph{When is there a unique socle-vector associated to a given $h$-vector?}, Comm.  Algebra  \textbf{34} (2006), no. 5, 1847--1860.

\bibitem{zanello1} F. Zanello: \emph{Stanley's theorem on codimension 3 Gorenstein $h$-vectors}, Proc. Amer. Math. Soc. \textbf{134} (2006),  no. 1, 5--8.

\bibitem{zanello3} F. Zanello: \emph{Interval Conjectures for level Hilbert functions}, J. Algebra \textbf{321} (2009), no. 10, 2705--2715.

\end{thebibliography}
\end{document}